\documentclass[12pt]{amsart}
\usepackage{amsfonts}
\usepackage{amsfonts,amssymb}

\usepackage{mathrsfs}
\let\mathcal\mathscr

\usepackage[all,ps,cmtip]{xy}

\def\Z{{\bf Z}}

\def\llra{\hbox to 10mm{\rightarrowfill}}

\def\lllra{\hbox to 15mm{\rightarrowfill}}

\def\PA{{\widehat A}}
\def\PB{{\widehat B}}
\def\PK{{\widehat K}}

\def\PY{{\widehat Y}}

\def\Pp{{\widehat\pi}}

\def\phi{{\varphi}}
\def\wA{{\widetilde A}}
\def\wf{{\widetilde f}}
\def\wX{{\widetilde X}}

\def\wK{{\widetilde K}}

\def\cI{\mathcal{I}}

\def\cF{\mathcal{F}}
\def\cL{\mathcal{L}}
\def\cO{\mathcal{O}}
\def\cG{\mathcal{G}}

\def\cH{\mathcal{H}}

\def\cS{\mathcal{S}}

\def\cJ{\mathcal{J}}

\def\cQ{\mathcal{Q}}

\DeclareMathOperator{\codim}{codim}
\DeclareMathOperator{\Pic}{Pic}
\DeclareMathOperator{\Bs}{Bs}

\DeclareMathOperator{\Supp}{Supp}
\DeclareMathOperator{\Exc}{Exc}

\DeclareMathOperator{\pr}{pr}

\newtheorem{lemm}{Lemma}[section]
\newtheorem{theo}[lemm]{Theorem}

\newtheorem*{conj*}{Conjecture}

\newtheorem{theAlph}{Theorem}[section]

\theoremstyle{definition}
\newtheorem{defi}[lemm]{Definition}
\newtheorem{rema}[lemm]{Remark}

\newtheorem{exam}[lemm]{Example}

\newtheorem{sett}[lemm]{Setting}

\theoremstyle{remark}
\newtheorem*{remark*}{Remark}
\newtheorem*{note*}{Note}

\def\moins{\mathop{\hbox{\vrule height 3pt depth -2pt
width 5pt}\,}}
\newcommand{\res}[2]{\left.#1\right|_{#2}} 

\begin{document}
\title[On the Iitaka fibration]{On the Iitaka fibration of varieties of maximal Albanese dimesion}
\author{Zhi Jiang}
\address{Math\'ematiques B\^{a}timent 425\\
Universit\'{e} Paris-Sud\\
F-91405 Orsay, France}
\email{zhi.jiang@math.u-psud.fr}

\author{Mart\'{\i} Lahoz}
\address{Mathematisches Institut\\
Universit\"at Bonn\\
Endenicher Allee 60\\
53115 Bonn, Germany}
\email{lahoz@math.uni-bonn.de}

\author{Sofia Tirabassi}
\address{Dipartimento di Matematica\\ Universit\`{a} degli Studi Roma TRE\\
Largo S. L. Murialdo 1\\
 00146 Roma, Italy}
\email{tirabass@mat.uniroma3.it}


\date{\today}

\begin{abstract}
We prove that the tetracanonical map of a variety $X$ of maximal Albanese dimension induces the Iitaka fibration.
Moreover, if $X$ is of general type, then the tricanonical map is birational.
\end{abstract}

\maketitle
\section{Introduction}
In the study of smooth complex algebraic varieties, the natural maps defined by differential forms are of special importance.
In particular pluricanonical maps have been extensively studied.

When $X$ is of maximal Albanese dimension, the study was started by Chen and Hacon \cite{CH2, CH}.
Their results were improved by Jiang \cite{j2}, who showed that the $5$-th canonical map induces the Iitaka fibration.
Recently, Tirabassi \cite{ti}, showed that if $X$ is of general type then the $4$-th pluricanonical is birational.

Combining Theorems \ref{thm:trimap} and \ref{thm:tetraIitaka}, we obtain the following statement.
\begin{theAlph}\label{thmA} Let $X$ be a smooth projective variety of maximal Albanese dimension.
Then,
\begin{enumerate}
 \item the linear system $|4K_X|$ induces the Iitaka fibration of $X$,
 \item if $X$ is of general type, then the linear system $|3K_X|$ induces a birational map.
\end{enumerate}
\end{theAlph}
These bounds are optimal for varieties of maximal Albanese dimension.
On the one hand, if $X$ is a curve of genus $2$, then $|2K_X|$ is not birational (varieties of general type whose bicanonical map is not birational have been studied in \cite{BLNP,gv2}).
Furthermore, we produce varieties of dimension at least $4$, whose tricanonical map does not induce the Iitaka fibration (see Example \ref{ex:not3}).

We observe that when $\chi(X,\omega_X)>0$, the birationality of the tricanonical map was proved by Chen and Hacon \cite[Thm.~5.4]{CH}.
Hence, we restrict ourselves to the case $\chi(X,\omega_X)=0$.
In this situation we have a special fibration where the $m$-th pluricanonical linear system restricts surjectively to a general fibre for $m\geq 3$ (see Lemma \ref{lem:completeseries}).
On the base of this fibration, we construct two positive line bundles (see Lemmas \ref{lem:existencebundle} and \ref{lem:2isbirational}).
One of them induces a birational map $\phi$ on the base.
The other one is used to prove a non-vanishing result which shows that the fibration followed by $\phi$ factors through the pluricanonical map.
We use these line bundles to apply Lemma \ref{lem:1}, which allows us to proceed by induction on the dimension of $X$.
When $X$ is of general type, it is enough to consider the tricanonical map to get the non-vanishing mentioned above.
The lack of this non-vanishing
is what forces us to consider the tetracanonical map for non-general type varieties (as we note in Remark \ref{rem:after}).

\subsection*{Notation}
In the sequel, $X$ will always be a smooth complex projective variety of maximal Albanese dimension.
We denote by $a_X: X\rightarrow A_X$ the Albanese morphism.
Given an abelian variety $A$, let $\PA$ be its dual.
For a morphism $t: X\rightarrow A$ to an abelian variety and a coherent sheaf $\cF$ on $X$, we denote by $V^i(\cF, t)$ the $i$-th cohomological support loci:
\begin{equation*}\{P\in\Pic^0(A)\mid H^i(X, \cF\otimes t^*P)\neq 0\}.\end{equation*}
We will use the terms $M$-regular, I.T.-index $j$ ($IT^j$ for short) and continuously globally generated ({\em cgg} for short) as they are defined in \cite{pp1}. We will say that $\cF$ is a $GV$-sheaf if $\codim_{\Pic^0(X)} V^i(\cF,a_X)\geq i$ for all $i\geq 0$.
Given two divisors $E, F$, we will write $E\preceq F$ if $F-E$ is effective.

\section{Preliminaries}
We begin with some easy lemmas.

\begin{lemm}\label{lem:1}
Let $f: X\xrightarrow{g} Z\xrightarrow{h} Y$ be fibrations between smooth projective varieties.
Let $L$ be an line bundle on $X$.
If the following two conditions hold:
\begin{itemize}
\item[1)] The image of $H^0(X, L)\rightarrow H^0(X_y, L|_{X_y})$ induces a map birationally equivalent to $\res{g}{X_y}: X_y\rightarrow Z_y$ for a general fibre $X_y$ of $f$;
\item[2)] There are line bundles $H_i$, $1\leq i\leq M$, on $Y$ such that $L-f^*H_i$ is effective and the multiple evaluation map
\begin{equation*}
\varphi_Y: Y\rightarrow \mathbf{P}(H^0(Y, H_1)^*)\times\cdots\times \mathbf{P}(H^0(Y, H_M)^*)
\end{equation*}
is birational.
\end{itemize}
Then, the linear system $|L|$ induces a map birationally equivalent to $g: X\rightarrow Z$.
\end{lemm}

\begin{proof}
Since the $L-f^*H_i$ are effective we have a map $\pi$ that induces the following diagram
\begin{equation*}
\xymatrix{X \ar@/_2pc/[dd]_f\ar[rr]^{\varphi_{|L|}}\ar[d]_g &&\mathbf{P}(H^0(X,L)^*) \ar@{>}[dd]^\pi\\
Z\ar@{.>}[urr]_{\varphi_Z}\ar[d]_h\\
Y\ar[rr]^(.25){\varphi_Y} && \mathbf{P}(H^0(Y, H_1)^*)\times\cdots\times \mathbf{P}(H^0(Y, H_M)^*)
}
\end{equation*}
Condition 2) guarantees that $\varphi_{|L|}$ separates generic
fibres of $f$ and condition 1) shows that the map $\varphi_{|L|}$
factorizes as $\varphi_Z\circ g$
and a general fibre of $h$ is mapped birationally via $\varphi_Z$.
\end{proof}

We will need the following lemma to ensure the birationality of $\varphi_Y$ in the previous lemma.

\begin{lemm}\label{lem:2} Let $\Pp: \PY\rightarrow Y$ be a generically abelian Galois covering between smooth projective varieties of maximal Albanese dimensions.
Denote by $G$ the Galois group of $\Pp$.
We denote $b_{\PY}=a_Y\circ \Pp: \PY\rightarrow A_Y$.
Assume that $V^0(\omega_{\PY}, b_{\PY})=\Pic^0(Y)$ and
\begin{equation*}\Pp_*\omega_{\PY}^2=\bigoplus_{\chi\in G^*}\cH_{\chi},\end{equation*}
where $\cH_{\chi}$ is the torsion-free rank-one sheaf corresponding to the character $\chi\in G^*$.

Then, there exists $\cH_{\chi_0}$ such that the multiple evaluation map
\begin{equation*}
\varphi_{P_1\cdots P_M}: Y\rightarrow \mathbf{P}(H^0(Y, \cH_{\chi_0}\otimes P_1)^*)\times\cdots\times \mathbf{P}(H^0(Y, \cH_{\chi_0}\otimes P_M)^*)
\end{equation*}
is birational for some $P_i\in\Pic^0(Y)$, $1\leq i\leq M$.
\end{lemm}

If $\Pp$ is an isomorphism, then Lemma \ref{lem:2} is contained in the proof of \cite[Thm.~4.4]{CH}.

\begin{proof} We first write
\begin{equation*}\Pp_*\omega_{\PY}=\bigoplus_{\chi\in G^*}\cL_{\chi}.\end{equation*}
where $\cL_0=\omega_Y$ and all $\cL_i$ are torsion-free rank-one sheaves on $Y$.

Since $V^0(\omega_{\PY}, b_{\PY})=\Pic^0(Y)$, we conclude that there exists $\chi_1$ such that $V^0(\cL_{\chi_1}, a_Y)=\Pic^0(Y)$.
Denote by $Z$ the closed subset where $\cL_{\chi_1}$ is not locally free.
Then, for any $P\in\Pic^0(Y)$, the line bundle $\cL_{\chi_1}^{\otimes 2}\otimes P$ is globally generated on the open dense subset
\begin{equation*}Y\moins \left(Z\cup \cap_{P\in\Pic^0(Y)}\Bs(|\cL_{\chi_1}\otimes P|)\right).\end{equation*}

Since there is the natural $G$-map of torsion-free sheaves on $Y$,
\begin{equation*}
(\Pp_*\omega_{\PY})^{\otimes 2}\rightarrow \Pp_*\omega_{\PY}^2,
\end{equation*}
if we take $\chi_0=\chi_1^2$, then we have an inclusion $\cL_{\chi_1}^{2}\hookrightarrow \cH_{\chi_0}$.
Hence, there is an open dense subset $U$ of $Y$ such that the sheaf $\cH_{\chi_0}|_{U}$ is locally free and for any $P\in\Pic^0(Y)$, $\cH_{\chi_0}\otimes P$ is globally generated over $U$, i.e.~$\res{\varphi_{P}}{U}$ is a morphism.

On the other hand, we consider $\Pp_*(\omega_{\PY}^2\otimes
\cJ(||\omega_{\PY}||))$.
Since $\cJ(||\omega_{\PY}||)\hookrightarrow\cO_{\PY}$ is $G$-invariant, we can write
\begin{equation*}\Pp_*(\omega_{\PY}^2\otimes
\cJ(||\omega_{\PY}||))=\bigoplus_{\chi\in G^*}\cH_{\chi}',\end{equation*}
where $\cH_{\chi}'$ is torsion-free sub-sheaf of $\cH_{\chi}$.
Moreover, we have
\begin{equation*}H^0(Y, \cH_{\chi}')\simeq H^0(Y, \cH_{\chi}),\end{equation*}
and
\begin{equation*}H^i(Y, \cH_{\chi}'\otimes P)=0,\end{equation*}
for any $P\in\Pic^0(Y)$ and $i\geq 1$.

Therefore, $\cH_{\chi_0}'\otimes P$ is globally generated and locally free on the open subset $U$ and $\Supp(\cH_{\chi_0}/\cH_{\chi_0}')$ is contained in
$Y\moins U$.
Now, let $V=U\moins \Exc(a_Y)$ and for any point $y\in V$, from the exact sequence
\begin{equation*}0\rightarrow \cI_y\otimes \cH_{\chi_0}'\rightarrow \cH_{\chi_0}'\rightarrow \mathbf{C}_y\rightarrow 0,\end{equation*}
we see that $a_{Y*}\big(\cI_y\otimes \cH_{\chi_0}'\big)$ is a $M$-regular sheaf, so it is {\em cgg} (see \cite[Prop.~2.13]{pp1}).
Hence for any $z\in V$ different from $y$, there exists $P\in\Pic^0(Y)$ such that $\cI_y\otimes \cH_{\chi_0}'\otimes P$ is globally generated on $z$.

This shows that for any two different points $y, z\in V$ there exists $P\in\Pic^0(Y)$ and a divisor $D_P$ in $|\cH_{\chi_0}\otimes P|$ such that $y\in D_P$ but $z\notin D_P$.
Therefore $\varphi_P(y)\neq \varphi_P(z)$.

We take $P_1,\ldots, P_M$ such that $\varphi_{P_1\cdots P_M}$ becomes stable, namely $\varphi_{P_1\cdots P_M}$ is birational equivalent to $\varphi_{P_1\cdots P_M P}$ for any $P\in\Pic^0(Y)$.
Then, $\varphi_{P_1\cdots P_M}$ is birational.
\end{proof}

The following lemma should be compared to \cite[Lem.~3.1]{CH1}.

\begin{lemm}\label{lem:K_X/Yeff}
Let $f: X\rightarrow Y$ be a surjective morphism between smooth projective varieties.
Assume that $X$ is of maximal Albanese dimension.
Then, $K_{X/Y}$ is effective.
\end{lemm}

\begin{proof}
We have the natural inclusion $f^*\Omega_Y^1\xrightarrow{i} \Omega_X^1$.
Denote by $\cF$ the saturation of $i(f^*\Omega_Y^1)$.
Then, $\det(\cF)-f^*K_Y$ is an effective divisor on $X$.
We then consider the exact sequence
\begin{equation*}
0\rightarrow \cF\rightarrow \Omega_X^1\rightarrow \cQ\rightarrow 0.
\end{equation*}
Since $X$ is of maximal Albanese dimension, $\Omega_X^1$ is generically globally generated and hence so is $\cQ$.

Hence $\det(\cQ)$ is also an effective divisor.
Hence $K_{X/Y}=\det(\cF)-f^*K_Y+\det(\cQ)$ is effective.
\end{proof}

The following lemma is used in \cite[Lem.~2.3]{j1} under the assumption that the plurigenera are positive (since $X$ is of maximal Albanese dimension, this condition is automatically satisfied).
We recall it here for easy reference.

\begin{lemm}\label{lem:dominates}
Suppose that $f: X \to Y$ is an algebraic fibre space between smooth projective varieties.
Assume that $Y$ is of general type.
Then, the Iitaka model of $(X,K_X+(m-2)K_{X/Y})$ dominates $Y$, for any $m\geq 2$.
\end{lemm}

\section{Positive bundles on the base and surjectivity of the restriction map to a fibre}
We will use the following definition that it is strongly related to Hypotheses 4.7 in \cite{BLNP}.
\begin{defi}Let $\cF$ be a coherent sheaf on an abelian variety $A$.
We say that $\cF$ is almost $M$-regular if $V^0(\cF)=\PA$, $\codim_{\PA}V^i(\cF)\geq i+1$, for $1\leq i\leq \dim A-1$, and $\dim V^{\dim A}(\cF)=0$.
\end{defi}

Let $X$ be a smooth projective variety of dimension $n$ and maximal Albanese dimension.
We know that the pushforward of the canonical bundle $a_{X*}\omega_X$ is a $GV$-sheaf but it often fails to be $M$-regular, which makes the tricanonical map difficult to study.
Hence we consider the set
\begin{equation*}\cS_X:=\{0<j<n\mid \textrm{$V^j(\omega_X, a_X)$ has a component of codimension $j$}\},\end{equation*}
which measures how far $a_{X*}\omega_X$ is from being almost $M$-regular.

\begin{sett}\label{set:basic}
Assume that $\cS_X$ is not empty.
We denote by $k$ the maximal number of $\cS_X$ and $Q+\PB\subseteq V^k(\omega_X,a_X)$ a codimension-$k$ component, where $Q$ is zero or a torsion element of $\Pic^0(X)\moins \PB$.
Let the following commutative diagram
\begin{equation*}
\xymatrix{
X\ar[r]^{a_X}\ar[d]_{f} & A_X\ar[d]^{\pr}\\
Y \ar[r]^{a_Y} & B
}
\end{equation*}
be a suitable birational modification of the Stein factorization of the composition $\pr\circ a_X$, such that $Y$ is smooth.
\end{sett}

\begin{lemm}\label{lem:existencebundle}
Assume $\cS_X$ is not empty, so we are in Setting \ref{set:basic}.
Then, for some birational model of $f: X\rightarrow Y$, there exists a line bundle $\cL$ on $Y$ such that $a_{Y*}\cL$ is almost $M$-regular, $V^0(\cL,a_Y)=\PB$, and $\cO_X(K_X+jQ)\otimes f^*\cL^{-1}$ has a non-trivial section for some $j\in\mathbf{Z}$.
Moreover,
\begin{itemize}
\item[1)] if $Q$ is trivial, we can take $\cL$ to be $\omega_Y$ and $j=0$;
\item[2)] if $Q\in \Pic^0(X)\moins \PB$, then we can take $\cL$ such that $a_{Y*}\cL$ is $M$-regular.
\end{itemize}
\end{lemm}

\begin{proof}
We know by \cite[Theorem 0.1]{gl} that the dimension of a general fibre of $f$ is $k$.

Assume first that $Q$ is trivial.
For $P_B\in \PB-\cup_j V^1(R^jf_* \omega_X, a_Y)\neq \emptyset$ (e.g.~\cite[Thm.~5.8]{PP-GV}),
\begin{align*}
0&< h^k(X,\omega_X\otimes a_X^*P_B)=h^k(X,\omega_X\otimes f^*P_B)\\
&=h^0(Y,R^k f_*\omega_X\otimes a_Y^*P_B) & \text{e.g.~{\cite[Prop.~3.14]{PP-GV}}}\\
&=h^0(Y,\omega_Y \otimes a_Y^*P_B) & \text{{\cite[Prop.~7.6]{kollar1}}}.
\end{align*}
By generic vanishing $\chi(Y,\omega_Y)=h^0(Y,\omega_Y \otimes a_Y^*P_B)$ for a general $P_B\in \PB$.
Hence $\chi(Y,\omega_Y)>0$ and $V^0(\omega_Y, a_Y)=\Pic^0(Y)$.

Moreover, the pull-back by $\pr$ of any codimension-$j$ component of $V^j(\omega_{Y}, a_{Y})$ is a codimension-$(j+k)$ component of $V^{j+k}(\omega_{X},a_{X})$ (by \cite[Thm.~3.1]{kollar2}).
Hence by the maximality of $k$, we know that $$\codim_{\PB}V^i(\omega_Y, a_Y)\geq i+1$$ for all $0<i<\dim Y$ and $a_{Y*}\omega_Y$ is almost $M$-regular.
By Lemma \ref{lem:K_X/Yeff}, $\cO_X(K_X-f^*K_Y)$ is effective.

Now, assume that $Q\in \Pic^0(X)\moins \PB$.
We may choose $Q$ such that $G\cap \PB=0$, where $G:=\langle Q\rangle$ is the subgroup generated by $Q$.\\

\paragraph{\bf Claim} Up to a birational modification of $X$ and $Y$, there exists a diagram
\begin{equation}\label{eq:diag-existence}
\xymatrix{
\wX\ar[r]^{\pi}\ar[d]^{\wf}\ar@/^2pc/[rr]_{b_{\wX}} & X\ar[r]^{a_X}\ar[d]^{f} & A_X\ar[d]^{\pr}\\
\PY\ar[r]^{\Pp}\ar@/_2pc/[rr]^{b_{\PY}} & Y \ar[r]^{a_Y} & B
}
\end{equation}
where all varieties are smooth and the vertical morphisms are fibrations (in particular, $\PY$ is a modification of the Stein factorization of $f\circ \pi$).
Moreover $\pi:\wX\to X$ is birational to $G$-cover, $\Pp: \PY\to Y$ is a generically $G$-cover, and $\wf$ is a $G$-equivariant morphism.We also have
\begin{equation}
\label{eq:defLi}
\Pp_*\omega_{\PY}=\oplus_i\cL_i,
\end{equation}
where $\cL_0=\omega_Y$ and all $\cL_i$ are torsion-free rank-one sheaves on $Y$.

\begin{proof}[Proof of the claim]
\renewcommand{\qedsymbol}{}
We first consider the \'{e}tale cover $\pi: \wX\rightarrow X$ induced by $G$.
Let $\PY$ be the Stein factorization of $f\circ \pi$.
By \cite[Thm. 5.8 and Prop. 3.14]{PP-GV}, $R^jf_*(\omega_X\otimes Q)$ are $GV$-sheaves for all $j\geq 0$.
Then by \cite[Thm. 3.1]{kollar2}, for $P_B\in \PB$ general $h^k(\omega_X\otimes Q \otimes a_Y^*P_B)= h^0(Y, R^kf_*(\omega_X\otimes Q)\otimes a_Y^*P_B)$.
Since $Q+\PB\subset V^k(\omega_X, a_X)$, $R^kf_*(\omega_X\otimes Q)\neq 0$.
If we denote by $X_y$ a general fibre of $f$, then we know that $\res{Q}{X_y}$ is $\cO_{X_y}$.
Hence, a general fibre of $\wf$ is isomorphic via $\pi$ to a general fibre of $f$.
So $G$ is also the Galois group of the field extension $k(\PY)/k(Y)$.
After modifications of $\wf: \wX\rightarrow \PY$, we may assume that $G$ acts also on $\PY$, $\wf$ is a $G$-equivariant morphism,
and $\Pp: \PY\rightarrow Y$ is a generically $G$-cover of smooth projective varieties.
Hence
\begin{equation*}
\Pp_*\omega_{\PY}=R^k(f\circ\pi)_*\omega_{\wX}=\oplus_i(R^kf_*(\omega_X\otimes Q^i))=\oplus_i\cL_i,
\end{equation*}
where $\cL_0=\omega_Y$ and all $\cL_i$ are torsion-free rank-one sheaves on $Y$ and we conclude the proof of the Claim.
\end{proof}

The same arguments as in the case where $Q$ is trivial show that $V^0(\PY, b_{\PY})=\PB$ and $\codim_{\PB}V^i(\omega_{\PY}, b_{\PY})\geq i+1$ for all $0<i<\dim \PY$.
Moreover, we have
\begin{eqnarray*}\ker(\PB\rightarrow \Pic^0(\PY))&=&\PB\cap\ker(\PA_X\rightarrow\PA_{\wX})\\
&=& \PB\cap \langle Q\rangle=0.\end{eqnarray*}
Hence for $i\neq 0$, $V^{\dim Y}(\cL_i, a_Y)=\emptyset$ and $a_{Y*}\cL_i$ is $M$-regular (in particular, $\chi(Y,\cL_i)>0$).

We can take a modification $\epsilon: Y'\rightarrow Y$ such that
\begin{equation}\label{eq:defLi2}
\cL_i':=\epsilon^*\cL_i
\end{equation}
is locally free for all $i$.
Moreover, we consider a birational model of diagram (\ref{eq:diag-existence}):
\begin{eqnarray}\label{eq:modification}
\xymatrix@C=8pt@R=8pt{
&\wX\ar[rr]^(.3){\pi}\ar'[d][dd]_{\wf}&& X\ar[dd]^(.7){f}\\
\wX'\ar[ur]\ar[rr]^(.3){\pi'}\ar[dd]_(.7){\wf'}&& X'\ar[ur]_{\eta}\ar[dd]^(.7){f'}\\
&\PY\ar'[r][rr]_(-.4){\Pp}&& Y \\
\PY'\ar[ur]\ar[rr]_(.3){\Pp'}&& Y'\ar[ur]_{\epsilon}}
\end{eqnarray}
where all varieties are smooth, all slanted arrows are birational modifications, and $\pi'$ is birational equivalent to the \'{e}tale cover induced by $G$.

Since we have $\Pp^*\cL_i\hookrightarrow \omega_{\PY}$, we conclude that $\omega_{\PY'}\otimes \Pp'^*\cL_i'^{-1}$ has a non-trivial section.
We know $K_{\wX'/\PY'}$ is an effective divisor by Lemma \ref{lem:K_X/Yeff}.
Hence,
\begin{align*}
0&< h^0(\wX', \omega_{\wX'}\otimes \wf'^*\omega_{\PY'}^{-1})\leq h^0(\wX', \omega_{\wX'}\otimes \wf'^*\Pp'^*\cL_i'^{-1})\\
&= \sum_j h^0(X', \omega_{X'}\otimes \eta^*Q^j\otimes f'^*\cL_i'^{-1}).
\end{align*}
So there exists $j$ such that $\omega_{X'}\otimes \eta^*Q^j\otimes f'^*\cL_i'^{-1}$ is effective.
Thus in $2)$, we can take $f': X'\rightarrow Y'$ to be the birational model of $f$ and take $\cL$ to be any $\cL_i'$ for $i\neq 0$.
\end{proof}

\begin{lemm}\label{lem:2isbirational}
Assume $\cS_X$ is not empty, so we are in Setting \ref{set:basic}.
Consider the birational model obtained in Lemma \ref{lem:existencebundle}.
Then, after modifying $f$ by blowing-up $Y$, there exists a line bundle $\cH$ on $Y$ and $i\in \mathbf{Z}$ such that $\cO_X(2K_X+iQ)\otimes f^*\cH^{-1}$ has a non-trivial section.
Moreover we can take $P_i\in\Pic^0(Y)$, $1\leq i\leq M$ such that the multiple evaluation map
\begin{equation*}
\varphi_Y: Y\rightarrow \mathbf{P}(H^0(Y, \cH\otimes P_1)^*)\times\cdots\times \mathbf{P}(H^0(Y, \cH\otimes P_M)^*)
\end{equation*}
is birational.
\end{lemm}

\begin{proof}
If $Q$ is trivial, we just take $\cH$ to be $\omega_Y^2$.
By Lemma \ref{lem:2}, we conclude.

If $Q\in \Pic^0(X)\moins \PB$, then we use the same notations as in the proof of Lemma \ref{lem:existencebundle}.
Notice that we can consider diagram \eqref{eq:diag-existence} obtained in the Claim.
Now, we apply Lemma \ref{lem:2} to $\Pp$ and we take $\cH$ to be the direct summand $\cH_{\chi_0}$ of $\Pp_*\omega_{\PY}^2$.

As in diagram (\ref{eq:modification}), we can take an appropriate model of $f$ and assume that $\cH$ is a line bundle.
Since $\Pp^*\cH\hookrightarrow \omega_{\PY}^2$, as in the proof of Lemma \ref{lem:existencebundle}, there exists an integer $i\in\mathbf{Z}$ such that
$\cO_X(2K_X+iQ)\otimes f^*\cH^{-1}$ is effective.
\end{proof}

\begin{lemm}\label{lem:completeseries}
Assume $\cS_X$ is not empty, so we are in Setting \ref{set:basic}.
Let $\cL$ be the line bundle obtained in Lemma \ref{lem:existencebundle}.
Then, for $y$ a general point of $Y$, the restriction map
\begin{equation*}
H^0(X,\cO_X(mK_X-(m-3)f^*\cL)\otimes P)\rightarrow H^0(X_y, \cO_{X_y}(mK_{X_y})\otimes P)
\end{equation*}
is surjective, for any $m\geq 2$ and $P\in V^0(\omega_X^m, a_X)$.
\end{lemm}

\begin{proof}
We just prove the statement for $P=\cO_X$, the same argument works for any $P\in V^0(\omega_X^m, a_X)$.

There are two distinguished cases, whether $Q$ is trivial or not, which we address with slightly different techniques.

\textbf{Case A}.
Assume that $Q$ is trivial.

We have seen in the proof of Lemma \ref{lem:existencebundle} that $\chi(Y,\omega_Y)>0$, so $Y$ is of general type.
By Lemma \ref{lem:dominates}, the Iitaka model of $(X,K_X+(m-2)K_{X/Y})$ dominates $Y$ and by \cite[Lem.~2.1]{j1} there exists an asymptotic multiplier ideal sheaf $\cI:=\cJ(||K_X+(m-2)K_{X/Y}||)$ on $X$ such that $a_{Y*}f_*(\cO_X(2K_X+(m-2)K_{X/Y})\otimes\cI)$ is an $IT^0$ sheaf.
Hence, by \cite[Prop.~2.13]{pp1},
\begin{equation*}\cF:=f_*(\cO_X(2K_X+(m-2)K_{X/Y})\otimes\cI)
\end{equation*}
is {\em cgg} outside the exceptional locus of $a_Y$.

We conclude, similarly to \cite[Prop.~4.4 and Cor.~4.11]{BLNP} that, $\cF\otimes \omega_Y$ 
is globally generated in an open dense subset of $Y$.
Indeed, we first notice that $|\omega_Y\otimes P|$ is not empty for all $P\in\PB$.
Take $P_i$, $1\leq i\leq N$ such that the evaluation map
\begin{equation*}\oplus_{i=1}^NH^0(Y, \cF\otimes P_i)\otimes P_i^{-1}\rightarrow \cF\end{equation*}
is surjective.
Then
\begin{equation*}\oplus_{i=1}^NH^0(Y, \cF\otimes P_i)\otimes H^0(Y, \omega_Y\otimes P_i^{-1})\otimes \cO_Y\rightarrow \cF\otimes \omega_Y\end{equation*}
is surjective over $Y\moins \cup_{1\leq i\leq N} \Bs(|\omega_Y\otimes P_i|)$.
Finally this evaluation map factors through $H^0(Y, \cF\otimes \omega_Y)\otimes\cO_Y\rightarrow \cF\otimes\omega_Y$.

Moreover, by \cite[Lem.~3.5]{j1}, $\cF$ is a non-zero sheaf on $Y$ of rank $P_m(X_y)$.
Hence, over a general point $y\in Y$, $\cF\otimes k(y)$ is isomorphic to $H^0(X_y, \cO_{X_y}(mK_{X_y}))$.
Since
\begin{equation*}
\cF\otimes \omega_Y
\subset f_*(\cO_X(mK_X-(m-3)f^*K_{Y}))
\end{equation*}
and they have the same rank $P_m(X_y)$, we conclude the proof of the lemma when $Q$ is trivial.

\textbf{Case B}.
If $Q$ is non-trivial,
we use the same notation as in the proof of Lemma \ref{lem:existencebundle}.
After the modification performed in diagram \eqref{eq:modification}, we can assume that we have a birational model of $f:X\to Y$ such that there exists a diagram
\begin{equation*}
\xymatrix{
\wX\ar[r]^{\pi}\ar[d]^{\wf}\ar@/^2pc/[rr]_{b_{\wX}} & X\ar[r]^{a_X}\ar[d]^{f} & A_X\ar[d]^{\pr}\\
\PY\ar[r]^{\Pp}\ar@/_2pc/[rr]^{b_{\PY}} & Y \ar[r]^{a_Y} & B
}
\end{equation*}
where all varieties are smooth, the vertical morphisms are fibrations, and $\pi$ is birational to an \'etale cover.
Moreover, if $\Pp_*\omega_{\PY}=\oplus_i\cL_i$, then the line bundle $\cL$ constructed in Lemma \ref{lem:existencebundle} is a birational modification of  $\cL_i$ for $i\neq 0$.

We claim that the Iitaka model of $(X,(m-1)K_{X}-(m-2)f^*\cL)$ dominates $Y$.
Indeed, by Lemma \ref{lem:dominates}, the Iitaka model of $(\wX, (m-1)K_{\wX/\PY}+\wf^*K_{\PY})$ dominates $\PY$.
By definition of $\cL_i$ (see \eqref{eq:defLi}), we have $\Pp^*\cL_i\preceq K_{\PY}$, which implies that $\pi^*f^*\cL\preceq f^*K_{\PY}$ (recall \eqref{eq:defLi2} and $\cL=\cL'_i$ for some $i\neq 0$).
This implies that $(m-1)K_{\wX/\PY}+\wf^*K_{\PY}\preceq \pi^*((m-1)K_{X}-(m-2)f^*\cL) +E$, where $E$ is some $\pi$-exceptional divisor.
Since $\pi$ is birational to an \'etale cover, the claim is clear.

Then, by \cite[Lem.~2.1]{j1} there is an ideal $\cI$ of $X$ such that
\begin{equation*}
a_{Y*}(f_*\cO_X(mK_{X})\otimes \cI\otimes {\cL}^{-(m-2)})
\end{equation*}
is an $IT^0$ sheaf and by \cite[Proof of Lem.~3.9]{j1} $f_*\cO_{X}(mK_{X})\otimes \cI\otimes {\cL}^{-(m-2)}$ has rank $P_m(X_{y})$.
We conclude as before, that for $i\neq 0$, $f_*(\cO_{X}(mK_{X}-(m-3)f^*\cL))$ is globally generated over an open dense subset of $Y$.
\end{proof}


\begin{rema}\label{rem:compatible}
Observe that we can consider a birational model $f:X\to Y$ such that both $\cL$ and $\cH$ are line bundles in $Y$ that fulfil the desired properties listed in lemmas \ref{lem:existencebundle}, \ref{lem:2isbirational} and \ref{lem:completeseries}.

In particular, when in Setting \ref{set:basic} $Q$ is non-trivial, we also have the following diagram:
\begin{equation*}
\xymatrix{
\wX\ar[r]^{\pi}\ar[d]^{\wf}\ar@/^2pc/[rr]_{b_{\wX}} & X\ar[r]^{a_X}\ar[d]^{f} & A_X\ar[d]^{\pr}\\
\PY\ar[r]^{\Pp}\ar@/_2pc/[rr]^{b_{\PY}} & Y \ar[r]^{a_Y} & B
}
\end{equation*}
where all varieties are smooth, the vertical morphisms are fibrations, and $\pi$ is birational to an \'etale cover.
Moreover, $\omega_{\PY}\otimes \Pp^*\cL^{-1}$ and $\omega_{\PY}^2\otimes \Pp^*\cH^{-1}$ have non-trivial sections.
\end{rema}

\section{General type case}
In this section $X$ will be a variety of general type.

\begin{theo}\label{thm:trimap}Let $X$ be a smooth projective variety, of maximal Albanese dimension and general type.
Then, the linear system $|3K_X+P|$ induces a birational map, for any $P\in\Pic^0(X)$.
\end{theo}
\begin{proof}
We reason by induction on the dimension of $X$, that we will denote by $n$.
Note that for $n=1$ the result is well known.
So we assume that for any $P_Y\in\Pic^0(Y)$, $|3K_Y+P_Y|$ induces birational map for any smooth projective variety $Y$ of maximal Albanese dimension, general type, and $\dim Y\leq n-1$.

Observe that $\cS_X$ is empty if and only if $a_{X*}\omega_X$ is almost $M$-regular.
Since $X$ is of general type, if $\cS_X$ is empty, then $\chi(X,\omega_X)>0$ (e.g.~\cite[Prop.~4.10]{BLNP}).
We notice that Chen and Hacon have proved that $|3K_X+P|$ induces a birational map in this situation (see \cite[Thm.~5.4]{CH}).

From now on, we will assume $\cS_X$ is not empty.
As in the last section, we are in Setting \ref{set:basic}:
\begin{equation*}
\xymatrix{
X\ar[r]^{a_X}\ar[d]_{f} & A_X\ar[d]^{\pr}\\
Y \ar[r]^{a_Y} & B .
}
\end{equation*}
Consider an appropriate birational model of $f:X\to Y$ in the sense of Remark \ref{rem:compatible}.

Let $y\in Y$ be a general point and denote by $X_y$ a general fibre of $f$.
By Lemma \ref{lem:completeseries}, the restriction map
\begin{equation}\label{eq:complete} H^0(X, \cO_X(3K_X)\otimes P)\rightarrow H^0(X_y, \cO_{X_y}(3K_{X_y})\otimes P)
\end{equation}
is surjective for any $P\in\Pic^0(X)$ and, by the induction hypothesis,
\begin{equation*}|3K_{X_y}+ \res{P}{X_y}|
\end{equation*}
induces a birational map.

We have also produced interesting line bundles on $Y$ in Lemma \ref{lem:existencebundle} and Lemma \ref{lem:2isbirational}.
Let $\cH$ be the line bundle on $Y$ constructed in Lemma \ref{lem:2isbirational}.
According to Lemma \ref{lem:1}, in order to conclude the proof of the theorem, we just need to prove the following claim.\\

\textbf{Claim \ddag.} For every $P\in \PA_X$ and every $P'\in \PB$, the line bundle
\begin{equation*}
3K_X+P-f^*(\cH+ P')
\end{equation*}
has a non-trivial section.\\

Let be $\cJ:=\cJ(||2K_X-f^*\cH+\frac{1}{N}f^*H||)$, where $N$ is an integer large enough and $H$ is an ample divisor on $Y$.
For any $P\in\PA_X$, we define
\begin{equation*}
\cF_P:=f_*\big(\cO_X(3K_X-f^*\cH)\otimes \cJ\otimes P\big).
\end{equation*}
Observe that to conclude the proof of the claim it is enough to see that $V^0(\cF_P, a_Y)=\PB$.

For any ample divisor $H'$ on $Y$, we have that
\begin{equation}\label{eq:dag} \tag{\dag}
H^i(X, \cO_X(3K_X-f^*\cH)\otimes \cJ\otimes P\otimes f^*\cO_Y(H'))=0,
\end{equation}
for any $i>0$.
We postpone the proof of \eqref{eq:dag} to the end of the proof of this theorem.
From \eqref{eq:dag} we deduce that
\begin{equation*}
R^if_*\big(\cO_X(3K_X-f^*\cH)\otimes \cJ\otimes P)\big)=0,
\end{equation*}
for any $i>0$ (see e.g.~\cite[Lemma 4.3.10]{laz}).
Therefore,
\begin{equation*}
\chi(Y, \cF_P)=\chi(X, \cO_X(3K_X-f^*\cH)\otimes \cJ\otimes P)
\end{equation*}
is constant for $P\in\PA_X$.

By Lemma \ref{lem:existencebundle} and Lemma \ref{lem:2isbirational}, there exist integers $i$ and $j$ and effective divisors $D_1\in |K_X+iQ-f^*\cL|$ and $D_2\in |2K_X+jQ-f^*\cH|$.
Let $m=i+j$ and write $D=D_1+D_2\in |3K_X+mQ-f^*\cH-f^*\cL|$, i.e.
\begin{equation*}
H^0(X, \cO_X(3K_X+mQ-f^*\cH-f^*\cL))=H^0(X, \cO_X(D))\neq 0.
\end{equation*}

Since
\begin{align}\cJ&=\cJ(||2K_X-f^*\cH+\frac{1}{N}f^*H||)\notag\\
&\supset \cJ(||2K_X-f^*\cH||) &\textrm{$H$ is ample on $Y$}\notag\\
&= \cJ(||2K_X+jQ-f^*\cH||) &\textrm{$Q$ is torsion} \notag\\
&\supset \cO_X(-D_2) &\textrm{by \cite[Thm.~11.1.8]{laz}}, \notag
\end{align}
we have
\begin{align}H^0(Y, \cF_{mQ}\otimes\cL^{-1})&= H^0(X, \cO_X(3K_X+mQ-f^*\cH-f^*\cL)\otimes \cJ)\notag\\
&\supset H^0(X, \cO_X(D_1))\neq 0.\label{eq:wIeff}
\end{align}

Therefore, since $V^0(\cL,a_Y)=\PB$, we have $h^0(Y, \cF_{mQ}\otimes P')>0$ for all $P'\in\PB$.

On the other hand, we see by \cite[Lem.~2.5]{j1} that $\cF_P$ is a $GV$-sheaf for any $P\in\PA_X$.
Therefore, for $P'\in\PB$ general and any $P\in\PA_X$,
\begin{equation*}h^0(Y, \cF_{P}\otimes P')=\chi(Y, \cF_P)=\chi(Y, \cF_{mQ})=h^0(Y, \cF_{mQ}\otimes P')>0.
\end{equation*}
Hence, by semicontinuity, for any $P\in\PA_X$, $V^0(\cF_P, a_Y)=\PB$.

\begin{proof}[Proof of \eqref{eq:dag}]
We use the same notation as in the proof of Lemma \ref{lem:existencebundle} and Lemma \ref{lem:2isbirational}, which is summarized in Remark \ref{rem:compatible}.

Notice that $2K_{\wX/\PY}\preceq \pi^*(2K_X-f^*\cH)+E$, where $E$ is some $\pi$-exceptional divisor.
Since $K_{\wX/\PY}+\frac{1}{N}\wf^*\Pp^*H$ is a big $\mathbf{Q}$-divisor on $\wX$, then $2K_X-f^*\cH+\frac{1}{N}f^*H$ is a big $\mathbf{Q}$-divisor on $X$.
So \eqref{eq:dag} is a consequence of Nadel vanishing theorem (see \cite[Thm.~11.2.12]{laz}).
\end{proof}
 \renewcommand{\qedsymbol}{}
\end{proof}

\section{Iitaka fibration}
In this section $X$ will not necessarily be a variety of general type.

\begin{theo}\label{thm:tetraIitaka}Let $X$ be a smooth projective variety, of maximal Albanese dimension.
Then, the linear system $|4K_X+P|$ induces a model of the Iitaka fibration of $X$, for any $P\in V^0(\omega_X^2, a_X)$.
\end{theo}

Before starting the proof the Theorem \ref{thm:tetraIitaka}, which is parallel to the proof of Theorem \ref{thm:trimap}, let us fix the notation.

\begin{sett}\label{set:bas2} Consider the following diagram:
\begin{eqnarray*}
\xymatrix{
X\ar[r]^{a_X}\ar[d]^g& A_X\ar[d]^{\pr_Z}\\
Z\ar[r]^{a_Z}& A_Z}
\end{eqnarray*}
where $g: X\rightarrow Z$ is a model of the Iitaka fibration of $X$ such that $Z$ is smooth.
Let $K$ be the kernel of $\pr_Z$.
We denote by $X_z$ a general fibre of $g$, which is birational to its Albanese variety $\wK$, and the natural map $\wK\rightarrow K$ is an isogeny.
We know that $\pr_Z^*\PA_Z$ is an irreducible component of
\begin{equation*}\mathscr{K}:=\ker(\PA_X\rightarrow \Pic^0(X_z))\end{equation*}
and denote by $\cQ:=\mathscr{K}/\pr_Z^*\PA_Z$.
Observe that $\mathscr{Q}$ can be also identified with $\ker(\PK\rightarrow \widehat{\wK})$.
\end{sett}

\begin{rema}\label{rem:before}
The group $\cQ$ is often non-trivial and this is exactly the reason why the tricanonical map can not always induce the Iitaka fibration.
In some specific cases, given information about $\cQ$, we can prove that the tricanonical map or some twisted tricanonical map (the maps induced by $|3K_X+P|$ for some $P\in\Pic^0(X)$) will induce the Iitaka fibration (see Remark \ref{rem:after}).

Nevertheless we will construct a variety of maximal Albanese dimension (see Example \ref{ex:not3}), where NONE of the twisted tricanonical maps is birationally equivalent to the Iitaka fibration.
\end{rema}

Before proving the theorem, we start with an easy well-known observation.
We add its proof for the convenience of reader.
\begin{lemm}\label{lem:K}
The kernel $\mathscr{K}$ defined in Setting \ref{set:bas2}, satisfies
\begin{equation*}
\mathscr{K}=V^0(\omega_X^m,a_X) \qquad \text{ for all }m\geq 2.
\end{equation*}
\end{lemm}
\begin{proof} It is clear that $V^0(\omega_X^m,a_X)\subseteq \mathscr{K}$.
If $P \in \mathscr{K}$, then $g_*(\omega_X^m\otimes P)$ is a nontrivial torsion-free sheaf.
By \cite[Lem.~2.1]{j1}, $g_*(\omega_X^m\otimes \cJ(\|\omega_X^{m-1}\|)\otimes P)$ is an $IT^0$ sheaf for any $m\geq 2$.
Hence, we conclude since $0<h^0(Z, g_*(\omega_X^m\otimes \cJ(\|\omega_X^{m-1}\|)\otimes P))\leq h^0(X, \omega_X^m\otimes P)$.
\end{proof}

\begin{proof}[Proof of Theorem \ref{thm:tetraIitaka}]
We will prove the theorem by induction on the dimension of $X$.
We suppose the statement is true in dimension $\leq n-1$ and assume $\dim X=n$.
If $X$ is of general type, then we are back to Theorem \ref{thm:trimap}.
Hence we can assume $\kappa(X)=\dim Z=n-l$, for some number $l>0$.
In particular, $\cS_X$ is not empty.

Hence, we are in Setting \ref{set:basic}.
Let $k$ be the maximal number of $\cS_X$ and let $Q+\PB$ be an irreducible component of $V^k(\omega_X, a_X)$ of codimension $k$.
Since $V^k(\omega_X, a_X)\subseteq V^0(\omega_X,a_X)$, by \cite[Lem.~2.2]{CH2}, $\PB\hookrightarrow \pr_Z^*\PA_Z$ (recall Setting \ref{set:bas2}).
Hence, Setting \ref{set:basic} and \ref{set:bas2} combine in the following commutative diagram

\begin{eqnarray*}\label{eq:diag1}
\xymatrix{
X\ar[r]^{a_X}\ar[d]^{g}\ar@/_2pc/[dd]_f & A_X\ar[d]_{\pr_Z}\ar@/^2pc/[dd]^{\pr}\\
Z\ar[r]^{a_Z}\ar[d]^{h}& A_Z\ar[d]\\
Y \ar[r]_{a_Y} & B,
}
\end{eqnarray*}
where we choose an appropriate birational model of $f:X\to Y$ in the sense of Remark \ref{rem:compatible}.

Let $y\in Y$ be a general point and denote by $X_y$ and $Z_y$ general fibres of $f$ and $h$.
By Easy Addition Formula (e.g.~\cite[Thm.~10.4]{ii}), $\dim Y+\kappa(X_y)\geq \kappa(X)=\dim Z$.
Hence $\kappa(X_y)\geq \dim Z_y$ and thus $\res{g}{X_y}: X_y\rightarrow Z_y$ is the Iitaka fibration of $X_y$.

By Lemma \ref{lem:completeseries} and \ref{lem:K}, the restriction map
\begin{equation*} H^0(X, \cO_X(4K_X+ P))\rightarrow H^0(X_y, \cO_{X_y}(4K_{X_y}+\res{P}{X_y}))
\end{equation*}
is surjective, for any $P\in \mathscr{K}$.
Notice that $\res{P}{X_y}\in V^0(\omega_{X_y}^2, a_{X_y})$, so by induction hypothesis,
\begin{equation*}|4K_{X_y}+ \res{P}{X_y}|
\end{equation*}
induces the Iitaka fibration $\res{g}{X_y}: X_y\rightarrow Z_y$.

Let $\cH$ be the line bundle on $Y$ constructed in Lemma \ref{lem:2isbirational}.
Then, by Lemma \ref{lem:1}, we just need to prove the following claim to complete the proof of the theorem.\\

\textbf{Claim.} For every $P\in \mathscr{K}$ and every $P'\in \PB$,
\begin{equation*}
4K_X+P-f^*(\cH+ P')
\end{equation*}
has a non-trivial section.\\

Let be $\cJ:=\cJ(||3K_X-f^*\cH+\frac{1}{N}f^*H||)$, where $N$ is an integer large enough and $H$ is an ample divisor on $Y$.
For any $P\in\mathscr{K}$, we define
\begin{equation*}
\cG_P:=g_*\big(\cO_X(4K_X+P-f^*\cH)\otimes \cJ\big).
\end{equation*}
To conclude the proof of the claim it is enough to see that $V^0(h_*\cG_P, a_Y)=\PB$.

By \cite[Lemma 2.1]{j1}, we have
\begin{eqnarray*}
H^i(Z, \cG_P\otimes Q''\otimes h^*H')=0,
\end{eqnarray*}
for any $i\geq 1$, any ample divisor $H'$ on $Y$, and any $Q''\in \PA_Z$.
Hence,
\begin{equation*}
R^ih_*(\cG_P \otimes Q'')=0,
\end{equation*}
for any $i>0$ (see e.g.~\cite[Lemma 4.3.10]{laz}).
Therefore,
\begin{equation} \label{eq:const}
\chi(Y, h_*(\cG_P\otimes Q''))=\chi(Z, \cG_P\otimes Q'')
\end{equation}
is constant for $Q''\in\PA_Z$.

By Lemma \ref{lem:existencebundle} and Lemma \ref{lem:2isbirational}, we know there exists $m\in\mathbf{Z}$ such that
\begin{equation*}
H^0(X, \cO_X(3K_X+mQ-f^*\cH-f^*\cL))\neq 0.
\end{equation*}

Observe that $P-mQ$ is not necessarily in $\pr_Z^*\PA_Z$.
But, since $P-mQ\in \mathscr{K}$, we have that $a_{Z*}g_*\cO_X(K_X+P-mQ)$ is a non-trivial $GV$-sheaf.
In particular, $V^0(g_*\cO_X(K_X+P-mQ),a_Z)\neq \emptyset$.
Hence there exists $Q_0\in\pr_Z^*\PA_Z$ such that $P-mQ+Q_0\in V^0(\omega_X, a_X)$.

Therefore \begin{eqnarray*}&&4K_X+P+Q_0-f^*\cH-f^*\cL\\&=&(K_X+P-mQ+Q_0)+(3K_X+mQ-f^*\cH-f^*\cL)\end{eqnarray*} is the sum of two effective divisors.
By the same argument as in \eqref{eq:wIeff},
\begin{equation*}
H^0(Z, \cG_P\otimes h^*\cL^{-1}\otimes Q_0)\neq 0.
\end{equation*}
We know that $V^0(\cL,a_Y)=\PB$ (see Lemma \ref{lem:existencebundle}) and $h_*(\cG_P\otimes Q')$ is a $GV$-sheaf for any $P\in\mathscr{K}$ and any $Q'\in \PA_Z$ (see \cite[Lem.~2.5]{j1}).
Hence, for $P'\in\PB$ general,
\begin{align*}h^0(Y, h_*\cG_{P}\otimes P')&=\chi(Y, h_*\cG_P)= \chi(Y, h_*(\cG_P\otimes Q_0)) &\textrm{by \eqref{eq:const}}\\
&=h^0(Y, h_*(\cG_P\otimes Q_0)\otimes P')>0.
\end{align*}
By semicontinuity, $V^0(h_*\cG_{P}, a_Y)=\PB$ for any $P\in\mathscr{K}$.
\end{proof}

Now, we can make more precise Remark \ref{rem:before}.
\begin{rema}\label{rem:after}
In the previous proof, observe that if $P- mQ$ lies in $\pr^*_Z\PA_Z$, then $3K_X+P+Q_0-f^*\cH-f^*\cL$ is effective for some $Q_0\in \pr^*\PA_Z$.
So, we could have improved the result to the tricanonical map (assuming the induction hypothesis).
In particular, if $\mathscr{Q}:=\mathscr{K}/\pr_Z^*\PA_Z$ is trivial for $X$ and the successive fibres of the induction process, then the tricanonical map twisted by an element in $\mathscr{K}$ induces the Iitaka fibration.

Moreover, if for some $P\in\mathscr{K}$, $P+\pr_Z^*\PA_Z$ is an irreducible component of $V^0(\omega_X, a_X)$, then we can again prove that the tricanonical map twisted by an element in $\mathscr{K}$ induces the Iitaka fibration.
This shows that varieties of maximal Albanese dimension, where none of the twisted tricanonical map is birational equivalent to the Iitaka fibration, are closely related to varieties of maximal Albanese dimension, of general type with vanishing holomorphic Euler characteristic.
\end{rema}

We finish with an example of maximal Albanese dimension, whose tricanonical map does not induce the Iitaka fibration.
This example is based on the famous Ein-Lazarsfeld threefold, which is constructed in \cite[Ex.~1.13]{el} and further investigated in \cite{CDJ}.
\begin{exam}\label{ex:not3}
We take three bielliptic curves $C_i$ of genus $2$, $i=1,2,3$.
Let $\rho_i: C_i\rightarrow E_i$ be the double cover over an elliptic curve $E_i$ and denote by $\tau_i$ the involution of fibres of $\rho_i$.
We write
\begin{equation*}\rho_{i*}\cO_{C_i}=\cO_{E_i}\oplus \cL_i^{-1},\end{equation*}
where $\cL_i$ is a line bundle on $E_i$ of degree $1$.

Let $Y$ be the threefold $(C_1\times C_2\times C_3)/(\tau_1, \tau_2, \tau_3)$, which has only rational singularities.
We know that $a_Y: Y\rightarrow E_1\times E_2\times E_3$ is a $(\mathbf{Z}/2\mathbf{Z}\times \mathbf{Z}/2\mathbf{Z})$-cover.

We then take an abelian variety $A$ and a $(\mathbf{Z}/2\mathbf{Z}\times \mathbf{Z}/2\mathbf{Z})$-\'etale cover $\wA\rightarrow A$.
Set $\{\cO_A, P_1, P_2, P_3\}$ to be the kernel $\PA\rightarrow \widehat{\wA}$.

Denote $H=(\mathbf{Z}/2\mathbf{Z}\times \mathbf{Z}/2\mathbf{Z})$ and let $X'$ be the variety $(Y\times \wA)/H$, where $H$ acts diagonally on $Y\times \wA$.
Notice that $X'$ has only rational singularities and let $X$ be a resolution of singularities of $X'$.
The Albanese morphism
\begin{equation*}a_{X}: X\rightarrow E_1\times E_2\times E_3\times A\end{equation*}
is birationally a $(\mathbf{Z}/2\mathbf{Z}\times \mathbf{Z}/2\mathbf{Z})$-cover.

After permutation of $\{P_i, i=1,2,3\}$, we have
\begin{eqnarray*}a_{X*}\omega_{X}^3&\simeq&\big(\cL_1^2\boxtimes\cL_2^2\boxtimes \cL_3^2\boxtimes \cO_A\big)\oplus \big(\cL_1^3\boxtimes\cL_2^3\boxtimes \cL_3^2\boxtimes P_1\big)\\
&\oplus &\big(\cL_1^3\boxtimes\cL_2^2\boxtimes \cL_3^3\boxtimes P_2\big) \oplus \big(\cL_1^2\boxtimes\cL_2^3\boxtimes \cL_3^3\boxtimes P_3\big)
\end{eqnarray*}

It is easy to check that for any $P\in\Pic^0(X)$, the linear series $|3K_X+P|$ can not induce the Iitaka fibration $X\rightarrow E_1\times E_2\times E_3$.

Using, the notation of Setting \ref{set:bas2}, observe that
\begin{equation*}
\mathscr{K}=V^0(\omega_X^2,a_X)= \bigcup_{Q\in \{\cO_A,P_1,P_2,P_3\}} E_1\times E_2\times E_3\times \{Q\}
\end{equation*}
and $\cQ=\Z/2\Z\times \Z/2\Z$.
Indeed, $\cQ$ can be identified with $\{\cO_A,P_1,P_2,P_3\}$.
\end{exam}

\subsection*{Acknowledgements}
The first-named author enjoyed the hospitality of the Max Planck Institut for Mathematics in Bonn during the preparation of this paper.

The second-named author was supported by the SFB/TR 45 `Periods, moduli spaces, and arithmetic of algebraic varieties' and partially by the Proyecto de Investigaci\'on MTM2009-14163-C02-01.

The third-named author thanks her advisor, Prof.~G.~Pareschi, for having introduced her to the subject of pluricanonical maps and for many suggestions.
She is also deeply indebted to C.~Ciliberto for many mathematical conversations.
Finally she is grateful toward the department of Mathematics of Universit\`a di Roma ``Tor Vergata'' for its kind hospitality.

Finally, we want to thank the referees for carefully reading the paper and for many suggestions that improved the exposition.


\begin{thebibliography}{BLNP}
\bibitem[BLNP]{BLNP} Barja, M.A., Lahoz, M., Naranjo, J.C., and Pareschi, G., {\em On the bicanonical map of irregular varieties}, J.~Algebraic Geom. {\bf } (electronically 2011).
\bibitem[CDJ]{CDJ} Chen, J.A., Debarre, O., and Jiang, Z., {\em Varieties with vanishing holomorphic Euler characteristic}, preprint arXiv:1105.3418.
\bibitem[CH1]{CH1} Chen, J.A. and Hacon, C.D., {\em Characterization of abelian varieties},
Invent. math. {\bf 143} (2001), 435--447.
\bibitem[CH2]{CH2} Chen, J.A. and Hacon, C. D., {\em Pluricanonical maps of varieties of maximal Albanese dimension},
Math. Ann. {\bf 320} (2001), 367--380.
\bibitem[CH3]{CH} Chen, J.A. and Hacon, C.D., {\em Linear series of irregular varieties, Algebraic geometry in East Asia (Kyoto, 2001)}, 143--153, {\it World Sci. Publ.}, River Edge, NJ, 2002.
\bibitem[EL]{el} Ein, L. and Lazarsfeld, R., {\em Singularities of theta divisors and the
birational geometry of irregular varieties}, {J. Amer. Math. Soc.} {\bf 10}
(1997), 243--258.
\bibitem[GL]{gl} Green, M. and Lazarsfeld, R., {\em Higher obstructions to deforming cohomology groups of line bundles}, {J. Amer. Math. Soc.} {\bf 4} (1991), 87--103.
\bibitem[Ii]{ii} Iitaka, S.. {\em Algebraic Geometry: An Introduction to Birational Geometry of Algebraic
Varieties}. Springer, New York (1981).
\bibitem[J1]{j1} Jiang, Z., {\em An effective version of a theorem of Kawamata on the Albanese map}, {\it Communications in Contemporary Mathematics}
Vol. {\bf 13}, No. 3 (2011) 509--532.
\bibitem[J2]{j2} Jiang, Z., {\em On varieties of maximal Albanese dimension}, preprint arXiv:0909.4817, to appear in Manuscripta Math.
\bibitem[Ka]{kawamata} Kawamata, Y., {\em Characterization of abelian varieties}, { Compositio Math.} {\bf 43} (1981), no. 2, 253--276.
\bibitem[K1]{kollar1} Koll\'{a}r, J., {\em Higher Direct Images of Dualizing Sheaves I}, {Ann. of Math.} {\bf 123} (1986), no. 1, 11-42.
\bibitem[K2]{kollar2} Koll\'{a}r, J., {\em Higher Direct Images of Dualizing Sheaves II}, {Ann. of Math.} {\bf 124} (1986), no. 1, 171-202.
\bibitem[L]{gv2} Lahoz, M., {\em Generic vanishing index and the birationality of the bicanonical map of irregular varieties}, Math. Z. (electronically 2011).
\bibitem[La]{laz} Lazarsfeld, R., {\em Positivity in algebraic geometry II}, Ergebnisse der Mathematik und ihrer Grenzgebiete {\bf49}, Springer-Verlag, Heidelberg, 2004.
\bibitem[PP1]{pp1} Pareschi, G., and Popa, M., {\em Regularity on abelian varieties I},
{J. Amer. Math. Soc.} {\bf 16} (2003), 285--302.
\bibitem[PP2]{PP-GV} Pareschi, G., and Popa, M., {\em GV-sheaves, Fourier-Mukai transform, and generic vanishing},
{Amer. J. Math.} {\bf 133} (2011), 235--271.
\bibitem[Ti]{ti} Tirabassi, S., {\em On the tetracanonical map of varieties of general type and maximal Albanese
dimension},
preprint arXiv:1103.5236, to appear in Collect. Math.

\end{thebibliography}
\end{document}